\documentclass[preprint,12pt]{elsarticle}
\usepackage{amssymb}
\usepackage{amsfonts}
\usepackage{color}
\usepackage{graphics}
\usepackage{epsfig,psfrag,graphicx}
\usepackage{amssymb}
\usepackage{eepic,epic}

\usepackage{epsfig} 
\usepackage{amsmath} 
\usepackage{amssymb} 
\usepackage{amsbsy} 

\newtheorem{definition}{Definition}[section]
\newtheorem{theorem}{Theorem}[section]
\newtheorem{lemma}[theorem]{Lemma}

\newtheorem{proposition}[theorem]{Proposition}

\newcommand{\nc}{\newcommand}
\newcommand{\mA}{\mathcal{A}}
\nc{\V}{{\cal V}} \nc{\M}{{\cal M}} \nc{\T}{{\cal T}}
\def\u{u}
\nc{\D}{{\cal D}} \nc{\W}{\mathbb W} \nc{\ti}{\tilde}
\nc{\wti}{\widetilde} \nc{\vep}{\varepsilon}
\nc{\R}{{\mathbb R}} \nc{\N}{{\mathbb N}}

\nc{\di}{\displaystyle}

\nc{\pa}{\partial} \nc{\lra}{\longrightarrow}
 \nc{\weak}{\rightharpoonup}
\nc{\weakstar}{\stackrel{\ast}{\rightharpoonup}} \nc{\proof}{{\bf
Proof: }} \nc{\en}{^{\varepsilon,n}} \nc{\Rd}{{{\mathbb R}^{d}}}
\nc{\biting}{\stackrel{\,\,b}{\rightarrow}}

\renewcommand{\div}{{\mathrm{div}}\,}
\nc{\bB}{B}
\nc{\bS}{A}

\nc{\modular}[1]{{\stackrel{ #1}{\longrightarrow\,}}}
\newcommand{\Ref}[1]{{\rm(\ref{#1})}}
\nc{\vd}{\bar{v}} \nc{\zd}{\bar{z}}
\def\bbbone{{\mathchoice {\rm 1\mskip-4mu l}
{\rm 1\mskip-4mu l} {\rm 1\mskip-4.5mu l} {\rm 1\mskip-5mu l}}}
\begin{document}
\begin{frontmatter}
\title{Nonlinear parabolic problems in  Musielak--Orlicz spaces  }
\author{Agnieszka \'Swierczewska-Gwiazda\\[2ex]
}
\address{Institute of Applied Mathematics and Mechanics, \\University of Warsaw,\\Banacha 2, 02-097 Warsaw, Poland,\\phone: +48 22 5544551}
\begin{abstract}
Our studies  are directed to the  existence of weak solutions to a parabolic problem containing a multi-valued term. The problem is formulated in the language of maximal monotone graphs.  We assume that the growth and coercivity conditions of a nonlinear term are prescribed  by means of time and space dependent  $N$--function.  This results in formulation of the problem in generalized Musielak-Orlicz spaces. We are using  density arguments, hence an important step of the proof is a uniform boundedness of appropriate convolution operators in Musielak-Orlicz spaces. For this purpose we shall need to assume a kind of logarithmic H\"older regularity with respect to~$t$ and~$x$. 
\end{abstract}
\end{frontmatter}

\section{Introduction}
We concentrate on an abstract parabolic problem.
Let ${\mathcal A}$ be a maximal monotone graph satisfying the assumptions (A1)--(A5) formulated below. We look for $u:Q\to\R$ and $ A:Q\to\R^d$ such that 
\begin{align}\label{P1a}
u_t-\div A = f\quad&{\rm in}\ Q,\\
\label{P1aa}
(\nabla u,A)\in{\mathcal A}(t,x)\quad&{\rm in}\ Q,\\
\label{P2a}
u(0,x)=u_0\quad&{\rm in}\ \Omega,\\
\label{P3a}
u(t,x)=0\quad&{\rm on}\ (0,T)\times\partial\Omega.
\end{align}
where  $\Omega\subset\R^d$ is an open, bounded set with a ${\cal C}^1$ boundary $\partial \Omega$,  $(0,T)$ is the time interval with $T<\infty$, $Q:=(0,T)\times\Omega$  and $\mA(t,x)\subset\R^d\times \R^d$ satisfies the following assumptions for a.a. $(t,x)\in Q$

\begin{enumerate}
\item[{ (A1)}] $\mA$  comes through the origin.
\item[{ (A2)}] $\mA$  is a monotone graph, namely
$$
(A_1-A_2)\cdot (\xi_1-\xi_2) \ge 0 \quad \textrm{ for all } (\xi_1, A_1),(\xi_2,A_2)\in \mA(t,x)\,.
$$
\item[{ (A3)}] $\mA$  is a maximal monotone graph. Let $(\xi_2, A_2)\in \R^d \times \R^d$.
\begin{equation*}\begin{split}
&\textrm{If } ( {A_1} - A_2)\cdot( {\xi_1} - \xi_2) \geq 0 \quad \textrm{ for all }
({\xi_1}, A_1) \in \mA(t,x)\\&
 \textrm{ then } (\xi_2, A_2) \in \mA(t,x).
\end{split}\end{equation*}
\item[{ (A4)}]  $\mA$  is an  {\it $M-$ graph.} There are non-negative $k\in
L^1(Q)$,   $c_*>0$ and $N$-function $M$ such that
\begin{equation*} 
A \cdot \xi \geq -k(t,x) +c_*(M(t,x,|\xi|) + M^*(t,x,|A|)) 
\end{equation*}
 for all $ (\xi,A)\in\mA(t,x).$ By an $N-$function we mean that  $M:\bar Q\times\R_+\to\R_+$, $M(t,x,a)$
  is measurable w.r.t. $(t,x)$ for all $a\in\R_+$ and  continuous w.r.t. $a$ for a.a. $(t,x)\in\bar Q$,  convex in $a$, has superlinear growth, $M(t,x,a)=0$ iff $a=0$ and  
$$\lim_{a\to\infty}\inf_{(t,x)\in Q}\frac{M(t,x,a)}{a}=\infty.$$
 Moreover the conjugate function $M^*$ is defined as 
$$M^*(t,x,b)=\sup_{a\in{\mathbb R}_+}(b\cdot a-M(t,x,a)).$$
\item[{ (A5)}]  The existence of a measurable selection.  Either there is $\tilde A:Q\times \mathbb{R}^{d} \to
\mathbb{R}^{d}$ such that $(\xi,
\tilde A(t,x,\xi)) \in \mA(t,x)$ for all $\xi \in \Rd$ and
$\tilde A$ is measurable, 
 or there is $\tilde \xi:Q\times \mathbb{R}^{d} \to \mathbb{R}^{d}$ such
that $(\tilde\xi(t,x,A), A) \in \mA(t,x)$ for all $A \in \R^{d}$ and $\tilde\xi$ is measurable.
\end{enumerate}
%
%
We are interested in existence of weak solutions. As the graph $\mA$ depends on $t$ and $x$,
we wanted to include the possibility of the growth conditions which are  time- and space-dependent. Hence 
 the growth 
conditions are also prescribed by a $(t,x)-$dependent $N-$ function.
 The studies are directed to the case of full generality in the upper and lower growth  of an $N-$ function with respect to the last variable. The consequence of relaxing this dependence  is the assumption of  higher  regularity with respect to  $t$ and $x$. More precisely,  
we will assume  log-H\"older continuos dependence on $t$ and $x$ of the function  $M$ and its conjugate $M^*$, i.e.,  it is supposed to  satisfy  the following:
\begin{itemize}
\item[(M1)] there exists a constant $H>0$ such that for all $x,y\in \Omega, t,s\in[0,T], |x-y|+ |t-s|\le\frac{1}{2}$
\begin{equation}\label{log}
\frac{M(t,x,a)}{M(s,y,a)}\le a^\frac{H}{\ln\frac{1}{|t-s|+|x-y|}}
\end{equation}
for all $a\in\R_+$ and moreover 
for every  bounded measurable set $G\subset \bar Q$  and every $z\in\R_+$
\begin{equation}\label{int}
\int_G M(t,x,z)< \infty.
\end{equation}
We assume that the same conditions hold also for $M^*$ in place of~$M$. 
\end{itemize}

The presented framework extends the result presented in~\cite{GwSw2010} in few directions. First of all we formulate the problem including the inclusion in the system. This formulation allows for capturing the problem of implicit relation between $A$ and $\nabla u$, and also allows for description 
of discontinuous dependence of $A$ on $\nabla u$. This kind of approach for problems of fluid mechanics was presented in \cite{BuGwMaRaSw2012, BuGwMaSw2012, GwMaSw2007} and also for steady problems in \cite{BuGwMaSw2009}. The articles \cite{ BuGwMaSw2009, 
GwMaSw2007} concern the setting in $L^p$ spaces, whereas \cite{BuGwMaRaSw2012,BuGwMaSw2012} concern the formulation in Orlicz spaces. Abstract elliptic and parabolic systems including inclusions in $L^p$ setting were considered in 
\cite{GwZa2005, GwZa2005a,GwZa2007}. Another novelty lies in the function space of  solutions. Because of time and space dependent growth--coercivity  conditions we work in Musielak--Orlicz spaces with $(t,x)-$dependent modulars. 
Having the restriction on the growth of an $N-$function and/or its conjugate (in particular $\Delta_2-$condition\footnote{We say that an $N-$function $M$ satisfies $\Delta_2$ condition if there exists a nonnegative function $h\in L^1(Q)$ and a constant $c>0$ such that $M(t,x,2a)\le c M(t,x,a)+h(t,x)$ for all $a\in\R_+$ and a.a. $(t,x)\in\bar Q$.})
simplifies the limit passage from an approximate to original problem. 
The previous results omitting the assumption of $\Delta_2-$condition on the conjugate function and relying on density arguments treated the case of $(t,x)-$independent modulars. This was related 
with approximation properties in Orlicz spaces and consequently using the tools of modular convergence. Allowing for the space and time dependence of the modular requires the information on the uniform boundedness of convolution operators. In particular, the dependence of the modular on $t$ is related with crucial difficulties appearing in the approximation of time derivative. 
In \cite{GwSw2010} the anisotropic spaces were considered. Here we cannot follow the same scheme.
The reasons are more detaily clarified in Section~\ref{preliminaries}.
Therefore we omit the generality of anisotropic $N-$function, restrict ourselves to isotropic one.

The studies on parabolic equations in Orlicz spaces have been a topic for many years, starting from the work of Donaldson~\cite{Donaldson}
and with later results of Benkirane, Elmahi and Meskine, cf.~\cite{BeEl1999, ElmahiMeskine, ElMe2005}.
 All of them concern the case of classical spaces, namely
Orlicz spaces with an $N-$function dependent only on $|\xi|$ without the dependence on $(t,x)$.  
Our important goal is to omit any restriction on the growth of an $N-$func\-tion, in particular the $\Delta_2-$condition for an $N-$function and its conjugate. This results in a need of formulating the approximation theorem (Theorem~\ref{closures}) and extensively using the notion of modular convergence (the precise definition appears in a sequel). The fundamental studies in this direction are due to Gossez for the case of elliptic equations \cite{Gossez1, Gossez2}. The appearance of  $(t,x)$ dependence in an $N-$function requires the studies on the uniform boundedness of the convolution operator.  
The considerations on the problem with an  $x-$dependent modular formulated in Musielak--Orlicz--Sobolev space is due to Benkirane et al. \cite{Be2011}. The authors formulate an approximation theorem with respect to the modular topology.
A particular case of Musielak-Orlicz spaces with $x-$dependent modulars are the variable exponent spaces $L^{p(x)}$, see 
e.g.~\cite{DiHaHaRu2011} for a comprehensive summary. The issue of density of smooth functions in this kind of spaces was considered e.g. in~\cite{FaWaZh2006, Zh2004}.

Before defining weak solutions we will  collect the notation.  
By the generalized Musielak-Orlicz class ${\mathcal L}_M(Q)$ we mean 
 the set of all measurable functions
 $\xi:Q\to\R^{d}$ for which the modular  $$\rho_M(\xi)=\int_Q M(t,x,|\xi(t,x)|) \,dx\,dt$$ is finite. 
 By $L_M(Q)$ we mean the generalized Orlicz space which is the set of
all measurable functions
 $\xi:Q\to\R^{d}$ for  which $\rho_M(\alpha\xi)\to0$ as $\alpha\to
 0.$
This is a Banach space with respect to the Luxembourg norm
$$\|\xi\|_M=\inf\left\{\lambda>0 : \int_Q M(t,x,|\xi(t,x)|)
\,dx\,dt\le1\right\}$$ 
or the equivalent Orlicz norm
$$|||\xi|||_M=\sup\left\{\int_Q \eta\cdot\xi \,dx\,dt : \eta\in L_{M^\ast}(Q),\int_Q M^\ast(t,x,|\eta(t,x)|)
\,dx\,dt\le1\right\}.$$ 
By $E_M(Q)$ we denote the closure of all bounded functions in
$L_M(Q)$. The space $L_{M^\ast}(Q)$ is the dual space of
$E_M(Q)$.
 A sequence $z^j$ converges modularly to $z$ in $L_M(Q)$ if
there exists $\lambda>0$ such that
$$\rho_M\left(\frac{z^j-z}{\lambda}\right)\to0$$
which is  denoted   by $z^j\modular{M} z$ for the modular
convergence in $L_M(Q)$. 



We use the notation ${\cal C}_{\textrm{weak}}(0,T; L^{2}(\Omega))$ for the space of all functions which are in 
$L^\infty(0,T; L^{2}(\Omega))$ such that $(u,\varphi)\in{\cal C}([0,T])$ for all $\varphi\in{\cal C}(\bar\Omega)$. Moreover, by 
${\cal C}_c^\infty(D)$ we mean the set of all compactly supported in $D$ smooth functions.
\begin{definition}\label{d:1}
Assume that
$
u_0\in L^{2}(\Omega), f\in L^\infty(Q)$ and $\mA$ be a maximal monotone graph.
We say that $(u,A)$ is a weak solution to \eqref{P1a}-\eqref{P3a}  if
\begin{align*}
&u \in  {\cal C}_{\textrm{weak}}(0,T; L^{2}(\Omega)),\  \nabla u \in L_M(Q),\ 
\ A\in  L_{M^\ast}(Q)
\end{align*}
and
\begin{equation}
\begin{split}
\int_Q \left(-u \varphi_t +A
\cdot \nabla \varphi \right) dx dt+\int_\Omega u_0(x)\varphi(0,x) dx
=\int_Qf\varphi dxdt
\end{split}
\end{equation}
 holds  for all $\varphi\in{\cal C}_c^\infty((-\infty,T)\times\Omega)$ and 
\begin{align*}
\left ( \nabla u ((t,x)), A(t,x) \right ) \in \mA(t,x)  \textrm{ for a.a. }
(t,x) \in Q.
\end{align*}
\end{definition}
Below  the main result of the present paper is formulated. 


\begin{theorem}\label{main2}
Let $M$ be an $N$--function satisfying (M1) and let  $\mA$ satisfy
conditions (A1)--(A5). Given $f\in L^\infty(Q) $
and
 $u_0\in
L^2(\Omega)$ 
there exists a weak solution to \eqref{P1a}--\eqref{P3a}.

\end{theorem}
%
%

The paper is organized as follows: In Section~\ref{preliminaries} we collect some fine properties of Musielak-Orlicz spaces and shortly describe the procedure of preparing the boundary to further approximation. The details are moved to the Appendix. Section~\ref{closures} concentrates on the approximation theorem and Section~\ref{4} is devoted to the proof of the existence result, namely Theorem~\ref{main2}. The last short section contains some examples of functions captured by the desired framework. The paper is completed by the appendix, where we collect necessary facts for handling the multi-valued problem. Finally, we provide some comments for possible extensions or different approaches. 

\section{Preliminaries}\label{preliminaries}
\begin{lemma}
Let $M$ and $M^*$ be conjugated $N-$functions. Then  for all $\xi\in L_M(Q)$ and $\eta\in L_{M^*}(Q)$ the following inequalities hold:
\begin{enumerate}
\item H\"older inequality
\begin{equation}\label{hoelder}
\int_Q \xi \eta \,dx\,dt\le c\|\xi\|_M\|\eta\|_{M^*}.
\end{equation}
\item Fenchel-Young inequality
\begin{equation}\label{F-Y}
|\xi\cdot\eta|\le M(t,x,\xi)+M^*(t,x,\eta). 
\end{equation}
\end{enumerate}
\end{lemma}
For the proof   see~\cite{Musielak}. 

%

In the next section we will approximate the function having a zero trace on the boundary of $\Omega$. The standard procedure in the case of at least Lipschitz boundaries is to observe that the domain is equal to the sum of star-shaped domains and proceed with an appropriate partition of unity and scaling the function on star-shaped sets.  However proceeding with the partition of unity leads to the necessity of either using the Poincar\'e inequality or truncating the function. The first option needs an additional set of assumptions, since the Poincar\'e inequality in Musielak-Orlicz spaces is a non-trivial fact, cf.~\cite{Fa2012}. We recall more details in part C of  the appendix. The option of truncating the function, which was used also in \cite{GwSw2010} would need the integration by parts formula for truncations, which in the case of time-dependent modulars does not hold. For these reasons we use here a non-standard approximation method, which consists in constructing a mapping wh
 ich transfe
 rs the
  area near the boundary of $\Omega$ to its interior.

\begin{proposition}\label{psi-delta}
There exists a mapping $\Psi^\delta:\Omega\to\R^d$  such that
\begin{enumerate}
\item[(i)] there exists a constant $K_1>0$ such that
$$\inf\limits_{x\in\Omega, y\in\partial\Omega}|\Psi^\delta(x)-y|\ge K_1\delta,$$
\item[(ii)] there exists a constant $K_2>0$ such that
$$\sup\limits_{x\in\Omega}|\Psi^\delta(x)-x|\le K_2\delta.$$
\item[(iii)]
$$\sup\limits_{x\in\Omega}|\nabla \Psi^\delta(x)- {\bf 1}|\to0$$
as $\delta\to0$ and where ${\bf 1}$ is an identity matrix. 
\end{enumerate}
\end{proposition}

The construction of the mapping $\Psi^\delta$ and the proof of its properties is moved to the appendix, part A. 

\bigskip


The next lemma is an important tool for the approximation theorem presented in the next section. An analogous result in the case of standard procedure, namely division for star-shaped domains and only $x-$dependent modulars was presented by  Benkirane et al.~\cite{Be2011}, see also 
\cite{GwMiWr2012} for the extension to an anisotropic case.

\begin{lemma}\label{modular-topology}
Let  $S\in
{\cal C}^\infty_c(\R^{d+1})$,
$\int_{\R^{d+1}} S(\tau,y)\,dy\,d\tau=1$ and  $S(t,x)=S(-t,-x)$. We define
$ S_\delta(t,x):=1/\delta^{d+1}S(t/\delta,x/\delta).$ 
Consider the family of operators
\begin{equation}\label{Sdelta}
\begin{split}
 {\cal S}_\delta z(t,x):=\int_Q
  S_\delta(t-s,\Psi^\delta(x)-y)z\left(s,y\right)  \,dy\,ds.
\end{split}\end{equation}
Let an $N-$function satisfy condition (M1).
Then there exist a  constants $c>0$ (independent of $\delta$) such that 
 for every  $ z\in L_M(Q)$ the following estimate holds
\begin{equation}\label{cont2}
\int_Q M(t,x, |{\cal S}_\delta z(t,x))|)\,dx\,dt\le c\int_Q M(t,x,| z(t,x)|)\,dx\,dt.
\end{equation}

\end{lemma}

\begin{proof}
%
Extend $z\in L_M(Q)$ by zero in the neighbourhood of the boundary of $\Omega$ outside of the image of $\Psi^\delta$. Due to this procedure the convolution with a kernel $S_\delta$ shall not lose the information of zero trace on the boundary. 
 Let ${\cal S}_\delta z(t,x)$ be defined by~\eqref{Sdelta}. 
For every $\delta>0$ there exists $N=N(\delta)$ such that  a family of closed
cubes $\{D_{\delta,k}\}_{k=1}^N$ with disjoint interiors and the length of an edge equal to $\delta$ covers $\Omega$, i.e.
$\Omega\subset \bigcup_{k=1}^ND_{\delta,k}$. Then consider the family of cubes centered the same as $D_{\delta,k}$ with an edge of the length $2\delta$. We shall call this family $\{G_{\delta,k}\}$. Note that if $x\in D_{\delta,k}$, then there exist $2^d$ cubes 
$G_{\delta,k}$ such that $x\in G_{\delta,k}$. Then divide the interval $[0,T]$ for the 
subintervals of the length $\delta$, which we  call $I_{\delta,i}$. Moreover by $J_{\delta,i}$ we shall mean the intervals of the length $2\delta$, namely $((i-3/2)\delta,(i+1/2)\delta)$. 
  Hence 
%
\begin{equation}\begin{split}
\int_0^T\int_\Omega& M(t,x,|{\cal S}_\delta z(t,x)|)\,dx\\
&=\sum\limits_{i=1}^{[T/\delta]}\sum\limits_{k=1}^{N}
\int_{I_{\delta,i}\cap(0,T)}\int_{D_{\delta,k}\cap\Omega}M(t,x,|{\cal S}_\delta z(t,x)|)\,dx\,dt.
\end{split}\end{equation}
Define 
\begin{equation}
m_{i,k}^\delta(\xi):=\inf_{(t,x)\in (J_{\delta,i}\times G_{\delta,k})\cap Q}M(t,x,\xi)\le 
\inf_{(t,x)\in (I_{\delta,i}\times D_{\delta,k})\cap Q}M(t,x,\xi)
\end{equation}
and 
\begin{equation}
\alpha_{i,k}(t,x,\delta):=\frac{M(t,x,|{\cal S}_\delta z(t,x)|)}{m_{i,k}^\delta
(|{\cal S}_\delta z(t,x)|)}.
\end{equation}
Then obviously
\begin{equation}\begin{split}
\int_0^T\int_\Omega& M(t,x,|{\cal S}_\delta z(t,x)|)\,dx\,dt\\&=\sum\limits_{i=1}^{[T/\delta]}\sum\limits_{k= 1}^N
\int_{I_{\delta,i}\cap(0,T)}\int_{D_{\delta,k}\cap\Omega}\alpha_k(t,x,\delta)m_{i,k}^\delta
(|{\cal S}_\delta z(t,x)|) \,dx\,dt.
\end{split}\end{equation}
We shall now concentrate on the uniform estimates of $\alpha_{i,k}(t,x,\delta)$ for sufficiently small $\delta$ and 
$(t,x)\in I_{\delta,i}\times D_{\delta,k}$.  Without loss of generality one can assume that 
$\| z\|_M\le 1$. By H\"older inequality \eqref{hoelder} we obtain 
\begin{equation}\label{22}\begin{split}
&|{\cal S}_\delta z(t,x)|\\
&\le\frac{1}{\delta^{d+1}}\sup_{B(0,1)}|S(t,y)|\int_Q\left|\bbbone_{B(0,\delta)}(y) z(t-s,
\Psi^\delta(x)-y)\right| \,dy\\
&\le \frac{1}{\delta^{d+1}}\sup_{B(0,1)} |S(t,y)| \| z\|_{1}\le \frac{c}{\delta^{d+1}}\|z\|_M
\le \frac{c}{\delta^{d+1}}.
\end{split}\end{equation}


Let now $(t_i,x_k)$ be the point where the infimum of $M(t,x,\xi)$ is obtained in the set $J_{\delta,i}\times G_{\delta,k}$. Then by log-H\"older
regularity we have
\begin{equation}
\alpha_{i,k}(t,x,\delta)=\frac{M(t,x,| {\cal S}_\delta z(t,x)|)}{M(t_i,x_k,| {\cal S}_\delta z(t,x)|)}\le| {\cal S}_\delta z(t,x)| ^\frac{H}{\ln\frac{1}{|x-x_k|+|t-t_i|}}.
\end{equation}
And as $x\in D_{\delta,k}$ and $x_k\in G_{\delta,k}$ then   $|x-x_k|\le \delta\sqrt{d}$ and for 
$t\in I_{\delta,i}$ and $t_i\in J_{\delta,i}$ we have $|t-t_i|\le\delta$. Hence for sufficiently small $\delta$, e.g. $\delta<\frac{1}{2(\sqrt{d}+1)}$ we have
$$| {\cal S}_\delta z(t,x)| ^\frac{H}{\ln\frac{1}{|x-x_k|+|t-t_i|}}\le | {\cal S}_\delta z(t,x)|^\frac{H}{\ln\frac{1}{\delta(\sqrt{d}+1)}}.
$$ 
Further we use  \eqref{22} to estimate as follows again for $\delta<\frac{1}{2(\sqrt{d}+1)}$
\begin{equation}\begin{split}
 | {\cal S}_\delta z(t,x)|^\frac{H}{\ln\frac{1}{\delta(\sqrt{d}+1)}}&\le(c\delta^{-(d+1)})^\frac{H}{\ln\frac{1}
 {\delta(\sqrt{d}+1)}}\\
& \le c^\frac{H}{\ln2} \cdot(\sqrt{d}+1)^\frac{H(d+1)}{\ln 2}\cdot \left(e^{\ln \delta(\sqrt{d}+1)}\right)^\frac{(d+1)H}{\ln\delta(\sqrt{d}+1)}
\\&
 \le (\sqrt{d}+1)^\frac{H(d+1)}{\ln 2} \cdot c^\frac{H}{\ln2}\cdot e^{(d+1)H}:=C.
\end{split}\end{equation}
Consequently
\begin{equation}\label{osza}
\alpha_{i,k}(t,x,\delta)\le C.
\end{equation}
Define $\tilde M(t,x,\xi):=\max_{i,k}m_{i,k}^\delta(\xi)$ where the maximum is taken with respect to all the sets 
$J_{\delta,i}\times G_{\delta,k}$. One immediately observes that  $\tilde M(t,x,\xi)\le M(t,x,\xi)$ for all $(t,x)\in Q.$
Another observation concerns the behaviour of $\Psi^\delta$ on the sets $G_{\delta,k}$. Note that the mapping 
$\Psi^\delta$ only changes the shape of the sets which overlap with a neighbourhood of the boundary but does not change 
their number. 
Using the uniform estimate \eqref{osza} and Jensen inequality we have 
\begin{equation}\begin{split}
\int_Q &M(t,x,| {\cal S}_\delta z(t,x)|)dxdy\le C\sum\limits_{i=1}^{[T/\delta]}\sum\limits_{k=1}^N
\int_{I_{\delta,i}}\int_{D_{\delta,k}}m_{i,k}^\delta(| {\cal S}_\delta z(t,x)|) 
\,dx\,dt\\&
\le C\sum\limits_{i=1}^{[T/\delta]}\sum\limits_{k=1}^N\int_{B(0,\delta)}|S_\delta(y)|\,dy
\int_{J_{\delta,i}}\int_{\Psi^\delta(G_{\delta,k})}
m_{i,k}^\delta( | z(t,x)|)\,dx\,dt\\&
\le 2^{d+1}C\int_{Q}\tilde M(t,x,| z(t,x)|) \,dx\,dt
\le 2^{d+1}C\int_{Q} M(t,x,| z(t,x)|) \,dx\,dt
\end{split}\end{equation}
which completes the proof. 
%

\bigskip

The remaining part of this section contains some properties of sequences convergent in Musielak-Orlicz spaces. 

\end{proof}

\begin{lemma}\label{lem-dense}
Let ${\mathbb S}$ be the set of all simple, integrable functions on $Q$ and let $$\int_A M(t,x,|z|)\,dx\,dt<\infty$$
for every $z\in \R^d$ and measurable set $A$ of finite measure. Then ${\mathbb S}$ is dense with respect to the modular topology in $L_M(Q)$. 
\end{lemma}
For the proof  see \cite[Theorem 7.6]{Musielak}. 
\begin{lemma}\label{modular-conv}
Let $z^j:Q\to\R^d$ be a measurable sequence. Then
$z^j\modular{M} z$ in $L_M(Q)$ modularly if and only if
$z^j\to z$ in measure and there exist some $\lambda>0$ such that
the sequence $\{M(t,x,\lambda z^j)\}$ is uniformly integrable in $L^1(Q)$,
i.e.,
$$\lim\limits_{R\to\infty}\left(\sup\limits_{j\in\N}\int_{\{(t,x):|M(\lambda z^j)|\ge
R\}}M(t,x,\lambda| z^j|)dxdt\right)=0.$$
\end{lemma}
For the proof see \cite[Lemma 2.1]{GwSw2008}.
\begin{lemma}\label{uni-int}
Let $M$ be an  $N$--function and for all $j\in\N$ let  $$\int_Q
M(t,x,|z^j|)\,dx\,dt\le c.$$
Then the sequence $\{z^j\}$ is
uniformly integrable in $L^1(Q)$.
\end{lemma}
For the proof see \cite[Lemma 2.2]{GwSw2008}.
\begin{proposition}\label{product}
Let $M$ be an  $N$--function and $M^\ast$ its complementary
function. Suppose that the sequences $\psi^j:Q\to\R^d$ and
$\phi^j:Q\to\R^d$ are uniformly bounded in $L_M(Q)$ and
$L_{M^\ast}(Q)$ respectively. Moreover $\psi^j\modular{M}\psi$
modularly in $L_M(Q)$ and $\phi^j\modular{M^\ast}\phi$  modularly
in $L_{M^\ast}(Q)$. Then $\psi^j\cdot\phi^j\to\psi\cdot\phi$
strongly in $L^1(Q)$.
\end{proposition}
For the proof see \cite[Proposition 2.2]{GwSw2008}.

\bigskip

The next  lemma  is the main tool for showing that  the limits of appro\-xi\-ma\-te sequences are in the graph ${\cal A}$ provided that the graph is maximal monotone. This lemma in such a form was formulated in 
\cite{BuGwMaRaSw2012}. For completeness we provide here its proof in the appendix. See also \cite{Wr2010} for the single-valued case. 

\begin{lemma}\label{Minty2}
Let $\mathcal{A}$ be maximal monotone $M$-graph.
 Assume that there are
sequences $\{A^n\}_{n=1}^\infty$ and
$\{\nabla u^n\}_{n=1}^{\infty}$ defined on $Q$ such that 
the following conditions hold:
\begin{align}
(\nabla u^n,A^n)& \in \mathcal{A} &&\textrm{ a.e. in } Q,\\\label{1.26}
\nabla u^n &\weakstar \nabla u &&\textrm{ weakly-star in } L_M(Q),\\
A^n &\weakstar A &&\textrm{ weakly-star in }  L_{M^*}(Q),\label{1.27}\\
\limsup_{n\to \infty} \int_{Q}A^n \cdot \nabla u^n \,d x\,d t &\le \int_{Q} A \cdot \nabla u \, d z.\label{Ass}
\end{align}
Then 
\begin{equation}
(\nabla u(t,x),A(t,x))\in \mathcal{A}(t,x)\quad \textrm{ a.e. in } Q.\label{Minty2-2}
\end{equation}
\end{lemma}

\section{Approximation theorem}\label{closures}
The current section is devoted to the issue of approximating functions having zero trace on the boundary and gradients bounded in Musielak-Orlicz space by compactly supported smooth functions in modular topology. This will be a crucial fact in the existence proof, in particular showing the energy equality, which is necessary for the limit passage in nonlinear term. This kind of approximation theorem in case of classical Orlicz spaces was proved in~\cite{ElmahiMeskine}. Before formulating the theorem let us 
define the space $V_M$ as follows
\begin{equation}
V_{M^*}(Q)=\{\phi=\div \phi_i: \phi_i\in L_{M^*}(Q)\}.
\end{equation}

\begin{theorem}\label{Aproksymacja}
If $u\in  L^2(Q)$, $\nabla u\in L_M(Q)$ 
and $u_t\in V_{M^*}(Q)+L^2(Q)$, 
then there exists a sequence $v^\delta \in{\cal C}_c^\infty([0,T]\times\Omega))$
satisfying 
\begin{equation}\label{zb_w_M1}
\nabla v^\delta\modular{M} \nabla u\ \mbox{ modularly in}\  L_M(Q) 
\ \mbox{and}\  v^\delta\to u \ \mbox{strongly in }\ 
L^2(Q).
\end{equation}
 Moreover we   can write
\begin{equation}
\frac{\partial v^\delta}{\partial t}=\div v_A^\delta+v_f^\delta\quad\mbox{and}\quad
\frac{\partial u}{\partial t} =\div v_A+v_f
\end{equation}
with 
\begin{equation}\label{zb_w_M*1}
 v^\delta_A\modular{M^*}  v_A \ \mbox{ modularly in }\ L_{M^*}(Q)  
 \quad\mbox{and}\quad v^\delta_f\to v_f \mbox{ strongly in } L^2(Q).
 \end{equation}

\end{theorem}
\begin{proof}
Let us  
 define 
\begin{equation}\label{Sdelta2}
\begin{split}
 {\cal S}_\delta  u(t,x):=\int_Q
  S_\delta(s,y) u\left(t-s,\Psi^\delta(x)-y\right)  \,dy\,ds.
\end{split}\end{equation}
We will concentrate on showing that 
\begin{equation}\label{goal}
\nabla  {\cal S}_\delta ( u)\modular{M} \nabla u \quad
{\rm modularly\  in\ } L_M(Q)
\end{equation}
as $\delta\to0_+$.  
Consider the sequence of  simple functions
$\xi_n:=\sum_{j=1}^n \alpha_j^n \bbbone_{G_j}(t,x)$, where $\bigcup_{j\in\{1,\ldots,n\}} G_j=Q$, which converges to $ \nabla u $ modularly in $L_M(Q)$. 
The existence of such sequence is provided by Lemma \ref{lem-dense}.
Note that the sequence  $\xi_n$ does not have to be in a gradient form.  Let 
$B_\delta:=\{(s,y)\in Q\ :\ |s|+|y|<\delta\}$.
Then 
\begin{equation}
\begin{split}
&{\cal S}_\delta \xi_n(t,x)-\xi_n\\
=&\int_{B_\delta}  S_\delta(s,y)\sum\limits_{j=1}^n\left(\alpha_j^n\bbbone_{G_j}
(t-s, \Psi^\delta(x)-y)-
\alpha_j^n\bbbone_{G_j}(t, x)\right) \,ds\,dy.
\end{split}\end{equation}
Hence with help of the Jensen inequality and Fubini theorem we conclude 
\begin{equation}\label{proste}\begin{split}
&\rho_M\left( {\cal S}_\delta \xi_n(t,x)-\xi_n\right)\\
&=\int_Q M(t,x,|\int_{B_1}  S(s,y)\sum\limits_{j=1}^n(\alpha_j^n\bbbone_{G_j}
(t-\delta s, \Psi^\delta(x)-\delta y)\\&\qquad-
\alpha_j^n\bbbone_{G_j}(t, x)) \,ds\,dy|)\,dt\,dx\\
&\le \int _{B_1}  S(s,y) \int_Q  M(t,x, |\sum\limits_{j=1}^n\alpha_j^n(\bbbone_{G_j}
(t-\delta s, \Psi^\delta(x)-\delta y)\\&\qquad-
\bbbone_{G_j}(t, x))|) \,dt\,dx \,ds\,dy.
\end{split}\end{equation}
Observe that  $\ \sum\limits_{j=1}^n\alpha_j^n\left(\bbbone_{G_j}
(t-\delta s, \Psi^\delta(x)-\delta y)-
\bbbone_{G_j}(t, x)\right) \,dt\,dx) \}_{\delta>0}$ converges a.e. in $Q$ to zero as $\delta\to0_+$ and 
\begin{equation}\begin{split}\label{estimate2}
M(t,x,|\sum\limits_{j=1}^n\alpha_j^n\left(\bbbone_{G_j}
(t-\delta s,\Psi^\delta(x)-\delta y)-
\bbbone_{G_j}(t, x)|\right)\\
\le \sup\limits_{z\in\{-1,0,1\}}M(t,x, |\sum\limits_{j=1}^n\alpha_j^n z|)
\le M(t,x, \sum\limits_{j=1}^n|\alpha_j^n |)
\end{split}\end{equation}
holds.
By \eqref{int} the right-hand side of \eqref{estimate2} is integrable and hence with help of the Lebesgue dominated convergence theorem we conclude that 
\begin{equation}\label{zb-prostych}
\lim\limits_{\delta\to0_+}\rho_M\left( {\cal S}_\delta \xi_n(t,x)-\xi_n\right)=0.
\end{equation}
According to ~Lemma~\ref{lem-dense} there exists $\lambda_0>0$ such that  
\begin{equation}\label{dense}
\lim\limits_{n\to\infty}\rho_M\left(\frac{\xi_n-\nabla u}{\lambda_0}\right)=0.
\end{equation}
Before  concluding \eqref{goal} we estimate the following integral with help of H\"older inequality and 
using the convexity of $M$. Note that by Proposition~\ref{psi-delta} $(iii)$ we can choose $\lambda_1$ such that 
 the term   
$\|\nabla \Psi^\delta(x)-1\|_\infty$ is less than $\lambda_1$. Then
\begin{equation}
\begin{split}
&\rho_M\left(\frac{  \nabla {\cal S}_\delta (u)-{\cal S}_\delta(\nabla u)}{\lambda_1}\right)\\
&=\int_QM\left(t,x,\frac{1}{\lambda_1}| \int_Q S_\delta(s,y)\nabla u(t-s,\Psi^\delta(x)-y)(\nabla \Psi^\delta(x)-1)\,ds\,dy|
\right)\,dx\\
&\le \int_QM\left(t,x, \frac{\|\nabla \Psi^\delta(x)-1\|_\infty}{\lambda_1}\int_Q 
|S_\delta(s,y)\nabla u(t-s,\Psi^\delta(x)-y)|\,ds\,dy
\right)\,dx\\
&\le\frac{ \|\nabla \Psi^\delta(x)-1\|_\infty}{\lambda_1} \int_QM\left(t,x,\int_Q 
|S_\delta(s,y)\nabla u(t-s,\Psi^\delta(x)-y)|\,ds\,dy
\right)\,dx.
\end{split}
\end{equation}
The modular on the right hand side is uniformly bounded by Lemma~\ref{modular-topology}. Hence using 
Proposition~\ref{psi-delta} $(iii)$ we conclude that 
\begin{equation}\label{komutator}
\lim\limits_{\delta\to0_+}\rho_M\left(\frac{  \nabla {\cal S}_\delta (u)-{\cal S}_\delta(\nabla u)}{\lambda_1}\right)=0.
\end{equation}
From the convexity of the modular and  \eqref{cont2}, choosing $\lambda>0$ such that $\lambda\ge \lambda_1+2\lambda_0+1
$
we have

\begin{equation}
\begin{split}
\rho_M&\left(\frac{  \nabla {\cal S}_\delta (u)-\nabla u}{\lambda}\right)\\
&\le \frac{\lambda_0}{\lambda}\left[
\rho_M\left(\frac{ {\cal S}_\delta  (\nabla u)- {\cal S}_\delta( \xi_n)}{\lambda_0}\right)
+
\rho_M\left(\frac{\nabla u-\xi_n}{\lambda_0}\right)\right]\\
&\quad+\frac{1}{\lambda}\rho_M\left( {\cal S}_\delta( \xi_n)-\xi_n\right)
+\frac{\lambda_1}{\lambda}\rho_M\left(\frac{  \nabla {\cal S}_\delta (u)-{\cal S}_\delta(\nabla u)}{\lambda_1}\right)\\
&\le \frac{\lambda_0(1+c)}{\lambda}
\rho_M\left(\frac{\nabla u- \xi_n}{\lambda_0}\right)
+\frac{1}{\lambda}\rho_M\left( {\cal S}_\delta (\xi_n)-\xi_n\right)\\
&\quad+\frac{\lambda_1}{\lambda}\rho_M\left(\frac{  \nabla {\cal S}_\delta (u)-{\cal S}_\delta(\nabla u)}{\lambda_1}\right).
\end{split}\end{equation}
By \eqref{zb-prostych} and \eqref{dense}, passing first with $\delta\to0_+$ and then with 
$n\to \infty$  we conclude that $\lim\limits_{\delta\to0_+}\rho_M\left(\frac{ \nabla{\cal S}_\delta
( u)-\nabla u}{\lambda}\right)=0.$

In the second step we shall show that for this approximation the conditions on the time derivative are valid.  Let us 
write $\frac{\partial u}{\partial t} =\div v_A+v_f$ with $v_A\in L_{M^*}(Q)$ and $v_f\in L^\infty(Q)$. We will show there exists 
$\lambda_2$ and sequences $v_A^\delta, v_f^\delta$ such that 
\begin{equation}
\lim\limits_{\delta\to0_+}\rho_{M_*}\left(\frac{v_A^\delta-v_A}{\lambda_2}\right)=0
\end{equation}
and 
\begin{equation}
v_f^\delta\to v_f\quad \mbox{strongly in}\ L^2(Q).
\end{equation}
Observe that 
\begin{equation}
\begin{split}
\frac{\partial {\cal S}_\delta(u)}{\partial t}&=
{\cal S}_\delta \frac{\partial u}{\partial t}
=
{\cal S}_\delta(\div v_A+v_f)\\&={\cal S}_\delta \div(v_A)
-{\cal S}_\delta\nabla v_A+{\cal S}_\delta v_f\\&=\div({\cal S}_\delta(v_A))
+{\cal S}_\delta v_f-{\cal S}_\delta\nabla v_A
\end{split}\end{equation}
and repeating the same procedure as above now for an $N-$function $M^*$ we conclude that 
\begin{eqnarray}
{\cal S}_\delta(v_A)\modular{M^*} v_A\quad&\mbox{modularly in } \ L_{M^*}(Q),\\
{\cal S}_\delta  v_f\to v_f\quad &\mbox{strongly in }\ L^2(Q).
\end{eqnarray}
The last convergence is strong since the $N-$function $M(t,x,\xi)=|\xi|^2$ satisfies $\Delta_2-$condition and hence modular and strong topologies coincide. 
\qed\\
\end{proof}

\section{Proof of Theorem \ref{main2}. }\label{4}
We shall provide an approximation in two steps. First, from the multivalued function we choose a selection, which is then mollified. Further  the finite-dimensional problem is formulated by means of Galerkin method. 

Consider a selection of $\mA$, namely
$\tilde A:Q\times \R^{d}\to\R^{d}$   assigning to each  $\bB\in\R^{d}$
exactly one value $\tilde A(t,x,\bB)\in \R^{d}$ so that $(\bB,\tilde A(t,x,\bB))\in\mA$.
With help of this selection we approximate $A$ as follows:
\begin{equation}\label{Te}
 A^\varepsilon (t,x,\xi):=(\tilde A* K^\varepsilon )(t,x,\xi)=\int_{\R^d}\tilde A(t,x,\zeta)
 K^\varepsilon (\xi-\zeta)\, d\zeta,
\end{equation}
where $ K^\varepsilon (\xi)=\frac{1}{\varepsilon} K\left(\frac{\xi}{\varepsilon } \right),\,\varepsilon>0$ and  
$ K\in {\cal C}^\infty_c(\R^d)$ is  a mollification kernel, i.e., a radially symmetric function with support in a unit ball $B(0,1)\subset\R^d$  and $\int_{\R^d} K \, d \xi=1$. It is not difficult to observe, using the convexity of $M$ and $M^*$ and by means of the Jensen inequality, that the approximation $ A^\varepsilon$ satisfies a condition analogous to  (A4), namely
\begin{itemize}
\item[{ (A4)$_\varepsilon$}]   There are non-negative $k\in
L^1(Q)$,   $c_*>0$ and an $N$-function $M$ such that
\begin{equation*} 
A^\vep \cdot \nabla u \geq -k(t,x) +c_*(M(t,x,\nabla u) + M^*(t,x,A^\vep)). 
\end{equation*}

\end{itemize}

%

The finite dimensional approximate problem  is constructed by means of Galerkin method. 
The basis consisting of eigenvectors of the Laplace operator  is chosen and  by $u^{ \varepsilon,n}$ we mean the solution to the considered problem projected to $n$ vectors of the chosen basis, namely 
$u^{ \varepsilon,n}(t,x):=\sum_{i=1}^nc_i^{ \varepsilon,n}(t)\omega_i(x)$ which solves the following system
\begin{equation}\label{aproksymacja}
\begin{split}
(u^{ \varepsilon,n}_t, \omega_i)+(A^\varepsilon(t,x,\nabla u^{ \varepsilon,n}),\nabla \omega_i)
=\langle f,\omega_i\rangle, \qquad i=1,\ldots,n, \\
u^{ \varepsilon,n}(0)=P^nu_0
\end{split}\end{equation}
where $P^n$ denotes the orthogonal projection of $L^2(\Omega)$ on the ${\rm span}\,\{\omega_1,
\ldots,\omega_n\}.$
 Let $Q^s:=(0,s)\times \Omega$ with 
$0<s<T$.
In the standard manner 
we conclude that  for $0<s<T$
\begin{equation}\label{Galerkin}\begin{split}
\frac{1}{2}\|u^{ \varepsilon,n}(s)\|^2_{2}+\int_{Q^s} A^\varepsilon(t,x,\nabla u^{ \varepsilon,n})\cdot \nabla u^{ \varepsilon,n}\, dx\,dt\\=\frac{1}{2}
\|u^{ \varepsilon,n}(0)\|^2_{2}+\int_0^s\langle f,u^{ \varepsilon,n}\rangle dt
\end{split}\end{equation}
holds. 
We estimate the right-hand side as follows 
\begin{equation}\begin{split}
|\int_0^s\langle f,u^{ \varepsilon,n}\rangle dt|&\le \int_0^s\|f\|_{\infty}\|\u^{ \varepsilon,n}
\|_{1}dt\le
c \int_0^s\|f\|_{\infty}\|\nabla u^{ \varepsilon,n}\|_{1} dt\\
&\le c\|f\|_{\infty}\int_0^s\|\nabla u^{ \varepsilon,n}\|_{1}dt \\
&\le \frac{c_*}{2}\int_0^s\|\nabla  u^{ \varepsilon,n}\|_{1} dt+K(c_*,Q)\|f\|_{\infty}\\
&\le  \frac{c_*}{2} \int_Q M(t,x,\nabla u^{ \varepsilon,n})dxdt+K\|f\|_{\infty}.\\
\end{split}\end{equation}
Using { (A4)$_\vep$} allows to conclude
\begin{equation}\begin{split}
\sup\limits_{s\in(0,T)}\|u^{ \varepsilon,n}(s)\|_2^2&+c_*\int_QM(t,x,\nabla u^{ \varepsilon,n})
+M^*(t,x,A^\varepsilon(t,x,\nabla u^{ \varepsilon,n})) \,dx\,dt
\\&
\le c(\|u_0\|_2^2+\|f\|_{\infty}+\int_Q k \,dx\,dt
).\end{split}\label{energy-eps}
\end{equation}
From \eqref{aproksymacja} and \eqref{energy-eps} it follows that 
$c_i^{\vep,n}(t) $ is bounded in $L^\infty([0,T])$ and $\frac{d}{dt}c_i^{\vep,n}(t)$ is bounded in $L_{M^*}([0,T])$, hence uniformly integrable in $L^1([0,T])$ and there exists a monotone, continuous $L:\R_+\to\R_+$, with $L(0)=0$ such that
$$\left|\int_{s_1}^{s_2}\frac{d}{dt}c_i^{\vep,n}(t)\,dt\right|\le L(|s_1-s_2|)$$
and thus 
$$|c_i^{\vep,n}(s_1)-c_i^{\vep,n}(s_2)|\le L(|s_1-s_2|).$$
Hence by the Arzel\`a-Ascoli theorem we are able to conclude an existence of uniformly convergent subsequence 
$\{c_i^{\vep_k,n}\}.$
The limit passage with $\varepsilon\to 0$ is done on the level of finite-dimensional problem and follows the similar lines as in~\cite{BuGwMaSw2012}. Nevertheless, we include the details for completeness. 
 In a consequence of \eqref{energy-eps} there exists a subsequence (labelled the same) such that
   \begin{equation}\label{ep-to-zero}\begin{split}
  u^{ \varepsilon,n}\to  u^{ n}\quad &\quad\mbox{strongly in}\quad {\cal C}([0,T];{\cal C}^1(\overline\Omega)),\\
A^\varepsilon (\cdot,\cdot,\nabla u^{ \varepsilon,n})\weakstar A^{ n} \quad &\quad\mbox{ weakly-star  in}\quad L_{M^*}(Q),\\
u^{ \varepsilon,n}_t\weakstar u^{ n}_t\quad&\quad\mbox{ weakly-star  in}\quad L_{M^*}(Q).
\end{split}\end{equation}
With these convergences one immediately obtains 
\begin{equation}\label{Galerkin}
\begin{split}
(u^{ n}_t, \omega_i)+(A^{ n},\nabla \omega_i)
=\langle f,\omega_i\rangle, \qquad i=1,\ldots,n, \\
u^{ n}(0)=P^nu_0.
\end{split}\end{equation}
 To show that $( \nabla u^{ n}, A^{ n})\in{\cal A} $ we will use the equivalence
  of $(i)$ and $(ii)$ in  Lemma~\ref{LS*}. 
 Since $\tilde A$ is the selection of the graph, according to Lemma~\ref{LS*} (a2)  for all $\zeta, B\in\Rd$ and a.a. $(t,x)\in Q$ it holds
\begin{equation}\label{mon}
(\tilde A(t,x,\zeta)-\tilde A(t,x,B))\cdot(\zeta-B)\ge0.
\end{equation}
We shall  add and subtract the term $(\tilde A(t,x,\zeta)-\tilde A(t,x,B))\cdot \nabla u^{ \varepsilon,n}$ and
then integrate  with respect to the probability measure, which  has the
density $K^\varepsilon (\nabla u^{ \varepsilon,n}-\zeta)$ and obtain  that
\begin{equation}\begin{split}\label{mon1}
\int_{\Rd}&(\tilde A(t,x,\zeta)-\tilde A(t,x,B))\cdot (\nabla u^{ \varepsilon,n}-B)K^\varepsilon (\nabla u^{ \varepsilon,n}-\zeta)d\zeta\\
&\ge \int_{\Rd}(\tilde A(t,x,\zeta)-\tilde A(t,x,B))\cdot (\nabla u^{ \varepsilon,n}-\zeta)K^\varepsilon (\nabla u^{ \varepsilon,n}-\zeta)d\zeta.
\end{split}\end{equation}
Consider $|\zeta|\le \|\nabla u^{ \varepsilon,\eta}\|_\infty
+ \vep$. The difference $(\tilde A(t,x,\zeta)-\tilde A(t,x,B))$ can be estimated  by a constant
dependent only  on $\bB$. Hence from  \eqref{mon1} we conclude that
\begin{equation}\label{concl_mon}\begin{split}
\left(\int_{\Rd}\tilde A(t,x,\zeta)K^\varepsilon (\nabla u^{ \varepsilon,n}-\zeta)d\zeta-\tilde A(t,x,B)\right)\cdot(\nabla u^{ \varepsilon,n}-B)\ge\\
-C_n(B) \int_{\Rd}|\nabla u^{ \varepsilon,n}-\zeta|K^\varepsilon (\nabla u^{ \varepsilon,n}-\zeta)d\zeta.
\end{split}\end{equation}
Using the strong convergence \eqref{ep-to-zero}$_1$ we see that the right hand side of \eqref{concl_mon} vanishes as $\vep \to 0_+$ and  we get
\begin{equation}\begin{split}
\liminf\limits_{\vep\to0_+}\left(A^\vep(t,x,\nabla u^{ \varepsilon,n})-\tilde A(t,x,B)\right)\cdot(\nabla u^{ \varepsilon,n}-B)\ge0 \qquad \textrm{ for a.a. } (t,x)\in Q.
\end{split}\end{equation}
The strong convergence of  $\nabla u^{ \varepsilon,n}$ and weak-star  convergence of  
$A^\vep(t,x,
\nabla u^{ \varepsilon,n})$
yields that for all $B\in\Rd$ and for a.a. $(t,x)\in Q$
\begin{equation}
(A^{ n}-\tilde A(t,x,B)) \cdot (\nabla u^{ n}-B)\ge0 \,.
\end{equation}
Thus, Lemma \ref{LS*} yields that 
\begin{equation*}
(\nabla u^{ n}(t,x),A^{ n}(t,x))\in\mA(t,x) \quad \textrm{ for a.a. } (t,x)\in Q\,.
\end{equation*}
 In the next step we shall provide the estimates uniform with respect to $n$. In the same manner we conclude that 
\begin{equation}\begin{split}
\sup\limits_{s\in(0,T)}\|u^{ n}(s)\|_2^2+\int_QM(t,x,\nabla u^{ n})+M^*(t,x,A^{ n}) \,dx\,dt\\
\le c(\|u_0\|_2^2+\|f\|_{\infty}+\|k\|_1)
\end{split}\end{equation}
which implies there exists a subsequence (again labelled the same) such that 
  \begin{equation}\begin{split}
 \nabla u^{ n}\weakstar \nabla u\quad &\quad\mbox{weakly-star in}\quad L_M(Q),\\
  u^{ n}\weak  u\quad &\quad\mbox{weakly in}\quad L^1(0,T;W^{1,1}(\Omega)),\\
A^{ n}\weakstar A\quad &\quad\mbox{ weakly-star  in}\quad L_{M^*}(Q),\\
u^{ n}\weakstar u\quad&\quad\,\mbox{weakly-star in}\,\,L^\infty(0,T;L^2(\Omega)).\\
u^{ n}_t\weakstar u_t\quad&\quad\mbox{ weakly-star  in}\quad W^{-1,\infty}(0,T;L^2(\Omega)).\end{split}\end{equation}
After passing to the limit  in \eqref{Galerkin} we obtain the following limit identity
\begin{equation}\label{limit}
 u_t  -\div {A} =  f
\end{equation}
holding in a distributional sense. To conclude that $(\nabla u, A)\in\mA(t,x)$ we need to establish that 
\eqref{Ass} is satisfied and then apply Lemma~\ref{Minty2}. For this purpose we want 
%
to test  equation \eqref{limit} with $u$.  For this reason consider the prelongation of $u$ on $\Omega\times\R$ such that 
$\nabla u\in L_M(\R\times\Omega)$.
By  Theorem~\ref{Aproksymacja} there exists a sequence $v^j\in{\cal C}_c^\infty(\R\times\Omega
)$ such that 
\begin{equation}\label{zb_w_M}
\nabla  v^j\modular{M} \nabla u\ \mbox{ modularly in}\  L_M(Q) 
\ \mbox{and}\  v^j\to u \ \mbox{strongly in }\ 
L^2(Q)
\end{equation}
and we can write
\begin{equation}
\frac{\partial v^j}{\partial t}=\div v_A^j+v_f^j\quad\mbox{and}\quad
\frac{\partial u}{\partial t} =\div v_A+v_f
\end{equation}
with 
\begin{equation}\label{zb_w_M*}
 v^j_A\modular{M^*}  v_A \ \mbox{ modularly in }\ L_{M^*}(Q)  
 \quad\mbox{and}\quad v^j_f\to v_f \mbox{ strongly in } L^2(Q).
 \end{equation}
Although we cannot test \eqref{limit} directly with $u$, but we can test with $v^j$ and then pass to the limit with $j\to\infty$. Indeed, 
\begin{equation}\label{poch}
\begin{split}
 \left\langle  u,  \frac{\partial v^j}{\partial t}\right\rangle &= 
 \left\langle  u-v^j,  \frac{\partial v^j}{\partial t}\right\rangle
 +\left\langle  v^j,  \frac{\partial v^j}{\partial t}\right\rangle=: I^j_1+I^j_2.
\end{split}
\end{equation}
We  observe that for  $0<s_0<s<T$ it follows
\begin{equation}
I_2^j=\int_{s_0}^s\int_\Omega v^j \frac{\partial v^j}{\partial t}\;dx\;dt
=\frac{1}{2}\left[\int_\Omega (v^j)^2\;dx\right]_{s_0}^{s}=
\frac{1}{2}\left(\| v^j(s)\|_{L^2(\Omega)}-\|v^j(s_0)\|_{L^2(\Omega)}\right). 
\end{equation}
Hence passing to the limit with $j\to\infty$ we immediately observe that 
\begin{equation}
\lim\limits_{j\to\infty}I^j_2=\frac{1}{2}\left(\| u(s)\|_{L^2(\Omega)}-\|u(s_0)\|_{L^2(\Omega)}\right).
\end{equation}
In the limit the term $I^j_1$ vanishes, indeed
\begin{equation}
I^j_1=\int_Q (u-v^j) (\div v_A^j+v_f^j)\;dx\;dt=\int_Q (\nabla v^j-\nabla u) \;v_A^j+
(u-v^j)\;v_f^j \;dx\;dt.
\end{equation}
Since  \eqref{zb_w_M} and \eqref{zb_w_M*} hold, we conclude with help of Proposition~\ref{product} the convergence of the first product and
the second follows immediately. 

Passing to the limit with $j\to\infty$ in the remaining terms is obvious. 
We are aiming to show that the identity
\begin{equation}
\frac{1}{2}\|u(s)\|_2^2-\frac{1}{2}\|u_0\|_2^2+\int_{Q^s}A\cdot \nabla u \,dx\,dt=\int_{Q_s} fu \,dx\,dt, 
\end{equation}
is satisfied which according to \eqref{limit}  holds for some $0<s_0<T$, not necessarily  equal to zero. 

To pass to the limit with $s_0\to0$ we need to establish the weak continuity of
$u$ in $L^2(\Omega)$ with respect to time. For this purpose we consider  
the sequence $\{\frac{d\u^n}{dt}\}$ and provide the uniform estimates. 
Let
 $\varphi\in L^\infty(0,T;W^{r,2}_0(\Omega))$, $\|\varphi\|_{L^\infty(0,T;W^{r,2}_0)}\leq 1$, where  
 $r>\frac{d}{2}+1$. Observe that
	\begin{equation*}
	\left\langle\frac{du^n}{dt}, \varphi \right \rangle=
	\left\langle\frac{du^n}{dt}, P^n\varphi\right \rangle
	= -\int_\Omega
	 A^n\cdot \nabla(P^n\varphi)\,dx\\
	 +\int_\Omega f\cdot P^n\varphi\,dx.
	\end{equation*}
Since $\|P^n\varphi\|_{W^{r,2}_0}\le\|\varphi\|_{W^{r,2}_0}$ and 
$W^{r-1,2}(\Omega)\subset L^\infty(\Omega)$ we estimate as follows 
	\begin{equation}\label{dt1}
	\begin{split}
	&\Big{|}\int_0^T\int_\Omega  A^n\cdot \nabla(P^n\varphi)
	dxdt\Big{|}\le \int_0^T\|
	 A^n\|_{L^1(\Omega)}\|\nabla(P^n\varphi)\|_{L^\infty(\Omega)}dt\\
	&\le c\int_0^T\| A^n\|_{L^1(\Omega)}\|P^n\varphi\|_{W^{r,2}_0}dt
	\le c
	\| A^n\|_{L^1(Q)}\|\varphi\|_{L^\infty(0,T;W^{r,2}_0)}.
\end{split}	\end{equation}
The estimates for the term containing $f$ are obvious. 
Hence we conclude that $\frac{du^n}{dt}$ is bounded in
$L^1(0,T;W^{-r,2}(\Omega))$. 
From the energy estimates and Lemma~\ref{uni-int} we conclude existence of  
a monotone, continuous function  $L:\R_+\to\R_+$, with $L(0)=0$ which 
is  independent of $n$  and
	$$\int_{s_1}^{s_2}\| A^n\|_{L^1(\Omega)}
	\le L(|s_1-s_2|)$$
for any $s_1,s_2\in[0,T]$. Conseqently, estimate \Ref{dt1} provides that
	$$\left|\int_{s_1}^{s_2}\left\langle\frac{d u^n}{dt},\varphi\right\rangle dt\right|	\le
	L(|s_1-s_2|)$$
for all $\varphi$ with 
${\rm supp}\ \varphi\subset(s_1,s_2)\subset[0,T]$ and 
$\|\varphi\|_{L^\infty(0,T;W^{r,2}_0)}\le 1$. 
Since 
	\begin{equation}\begin{split}
	\|u^n(s_1)-u^n(s_2)\|_{W^{-r,2}}
	=
	\sup\limits_{\|\psi\|_{W^{r,2}_0}\le1}\left|\left\langle
	\int_{s_1}^{s_2}\frac{du^n(t)}{dt},\psi\right\rangle\right |
	\end{split}\end{equation}
then
	\begin{equation}
	\label{equi}
	\sup\limits_{n\in\N}\|u^n(s_1)-u^n(s_2)\|_{W^{-r,2}}\le L(|s_1-s_2|),
	\end{equation}
which provides that 	the family of functions $\u^n:[0,T]\to W^{-r,2}(\Omega)$ is equicontinuous. 
Together with a uniform bound in $L^\infty(0,T;L^2(\Omega))$ it yields that the sequence $\{\u^n\}$ is relatively compact 
in  ${\cal C}([0,T];W^{-r,2}(\Omega))$ and we have $\u\in {\cal C}([0,T];W^{-r,2}(\Omega))$.  Consequently we can choose a sequence 
$\{s_0^i \}_i$, 
$s_0^i \to 0^+$ as $i\to\infty$ such that 
%
%
	\begin{equation}
	\u(s_0^i){{\stackrel{i\to \infty}{\longrightarrow\,}}}\u(0)\quad\mbox{in }W^{-r,2}(\Omega).
	\end{equation}
The limit coincides with the weak limit of $\{\u(s^i_0)\}$ in $L^2(\Omega)$
and hence
%
%
we conclude
	\begin{equation}\label{liminfu0}
	\liminf\limits_{i\to\infty}\|\u(s_0)\|_{L^2(\Omega)}\geq\|\u_0\|_{L^2(\Omega)}.
	\end{equation} 
Consequently we obtain from \eqref{Galerkin} for any   Lebesgue point $s$  of $\u$ that 
	\begin{equation}\label{777}
	\begin{split}
	\limsup\limits_{n\to\infty}\int_{Q_s} A(t,x,\nabla\u^n)\cdot \nabla\u^n 
	& = 
	 \frac{1}{2} \|\u_0\|^2_{2} -
	\liminf\limits_{k\to\infty} \frac{1}{2}  \|\u^n(s)\|^2_{2}\\
	& \leq
	 \frac{1}{2} \|\u_0\|^2_{2} -
	 \frac{1}{2}  \|\u(s)\|^2_{2}\\
	 & {{\stackrel{(\ref{liminfu0})}{\leq}}}
	 \liminf\limits_{i\to \infty}\left( \frac{1}{2} \|\u(s^i_0)\|^2_{2} -
	 \frac{1}{2}  \|\u(s)\|^2_{2}\right)\\
	  &{=} \lim\limits_{i\to\infty} 
	  \int_{s^i_0}^s\int_\Omega A\cdot\nabla\u dxdt
	  = \int_{0}^s\int_\Omega A\cdot\nabla\u dxdt
	 \end{split}
	\end{equation}
	which is exactly \eqref{Ass} and hence Lemma~\ref{Minty2} completes the proof. 
\section{Examples}
As a basic example of an $M-$graph  captured by the described framework one can mention the graph of a function of a  variable exponent  with a $(t,x)-$dependent exponent, namely  an $N-$function $M(t,x,\xi)=\xi^p(t,x)$.  In such a case we require that 
$p:Q\to[1,\infty)$ is a measurable function such that there exists a constant $H>0$ such that for all 
$x,y\in \Omega, t,s\in[0,T], |x-y|+ |t-s|\le\frac{1}{2}$
\begin{equation}\label{log-p}
|p(t,x)-p(s,y)|\le \frac{H}{\ln\frac{1}{|t-s|+|x-y|}}
\end{equation}
holds. For more details on the appearance of this condition in the theory of variable exponent spaces we refer the reader 
to~\cite{DiHaHaRu2011}. When condition \eqref{log-p} is satisfied we can construct the $N-$functions of very slow or very rapid growth, e.g. $M_1(t,x,\xi)=(e^{\xi})^{p(t,x)}-1$ or  $M_2(t,x,\xi)= \xi^{p(t,x)}\ln(\xi+1).$ Since the problem was introduced in the language of maximal monotone graphs the presented framework also captures the case of jumps with respect to $\xi$, e.g.  
the following case is admissible
\begin{equation}
M(t,x,\xi)=\left\{
\begin{array}{rlc}
M_2(t,x,\xi)&{\rm for}&\xi<1,(t,x)\in Q,\\[1ex]
M_1(t,x,\xi)&{\rm for}&\xi>1,(t,x)\in Q,\\[1ex]
[\ln 2,e^{p(t,x)}-1]&{\rm for}&\xi=1,(t,x)\in Q.
\end{array}\right.
\end{equation}


\appendix
\section{Domain $\Omega$}
\noindent
{\bf Proof of Proposition \ref{psi-delta}.}
Since $\Omega$ is a bounded domain, then $\bar\Omega$ is a compact set. Let 
$\{\Phi_\alpha, \alpha\in I\}$ be the atlas of $\bar\Omega$ and define by $U_\alpha:={\rm dom} (\Phi_\alpha)$. The sets $U_\alpha$ are open in $\R^d$. Let us now choose the sets which have nonempty intersection with a boundary, we can number them $\alpha=1,\ldots,\ell$, hence 
for these $\alpha's$ we have $U_\alpha\cap \partial \Omega\neq\emptyset$ and  $\partial \Omega\subset 
\bigcup_{\alpha=1,\ldots,\ell} U_\alpha$.
The boundary of  $\Omega$ is a  $C^1-$submanifold, and since it is a compact set, then 
without loss of generality we may assume that  for $\alpha=1,\ldots,\ell$ each  $\Phi_\alpha(U_\alpha)$ is a  ball of the same radius and  $\Phi_\alpha(\partial\Omega\cap U_\alpha)$ is the intersection of a ball 
$\Phi_\alpha(U_\alpha)$ with a hyperplane, which divides  the ball into two halves. Moreover assume the ball is centered at the origin and the hyperplane is orthogonal to the basis vector of $\R^d$, namely $e_d=(0,\ldots,0,1)$.
To construct the mapping $\Psi_1^\delta$ first we map $U_1$ for the set $\Phi(U_1)$. 
For simplicity let us use the notation 
$(x_1, \ldots, x_{d-1},x_{d})=(x',x_d).$ First define  a nonnegative function
\begin{equation}\label{T}
T_\epsilon(x',0):=\left\{
\begin{array}{rcl}
1&{\rm for}&\sqrt{\sum_{i=1}^{d-1}|x_i|^2}\le 1-\epsilon,\\[1ex]
{\rm smooth}&{\rm for}&\sqrt{\sum_{i=1}^{d-1}|x_i|^2}\in(1-\epsilon,1),\\[1ex]
0&{\rm for}&\sqrt{\sum_{i=1}^{d-1}|x_i|^2}= 1
\end{array}\right.
\end{equation}
and such that $\nabla T_\epsilon$ on the set  ${\sum_{i=1}^{d-1}|x_i|^2}= 1$ is equal to zero. 
Also without loss of generality we may assume that for each $x\in\partial \Omega$ there exists
$\alpha\in\{1,\ldots,\ell\}$ such that $T_\epsilon(\Phi_\alpha(x))=1$ and 
$(x',T_\epsilon(\Phi_\alpha(x))
\in\Phi_\alpha(U_\alpha)$ for each $x\in U_\alpha$.
Then for $1>\delta>0$ the mapping $\gamma^\delta:\Phi_\alpha(U_\alpha)\to\Phi_\alpha(U_\alpha)$ is defined by
\begin{equation}\label{gamma}
\gamma^\delta(x',x_d)=
\left\{\begin{array}{lcl}
(x',T_\epsilon(x',0)+(1-\delta)(x_d-T_\epsilon(x',0))&{\rm for}&x_d-T_\epsilon(x',0)<0,\\[1ex]
(x',x_d)&{\rm for}&x_d-T_\epsilon(x',0)\ge0.
\end{array}\right.
\end{equation}
Now we are ready to start to construct the function $\Psi^\delta$ which will be a composition 
of consequent mappings. First
define 
\begin{equation}\label{Psi1}
\Psi_1^\delta(x)=\left\{\begin{array}{lcl}
\Phi_1^{-1}(\gamma^\delta(\Phi_1(x)))&{\rm for}&x\in U_1,\\
x&{\rm for}& x\in\Omega\setminus U_1,
\end{array}\right.
\end{equation}
and
\begin{equation}\label{Psi2}
\Psi_\alpha^\delta(x)=\left\{\begin{array}{lcl}
\Phi_\alpha^{-1}(\gamma^\delta(\Phi_\alpha(\Psi_{\alpha-1}(x))))&{\rm for}&x\in U_\alpha,\\[1ex]
\Psi_{\alpha-1}(x)&{\rm for}& x\in\Omega\setminus U_\alpha.
\end{array}\right.
\end{equation}
Finally $\Psi^\delta:=\Psi^\delta_\ell.$
One can easily observe that 
\begin{equation}
\sup_{y\in\Phi_\alpha(U_\alpha)}\nabla \gamma^\delta(y)\to{\bf 1}
\end{equation}
and 
\begin{equation}
\sup_{y\in\Phi_\alpha(U_\alpha)} |\gamma^\delta(y)-y|\le\delta.
\end{equation}
The last one immediately implies the property $(iii)$ of the proposition. 
To conclude $(i)$ and $(ii)$ we shall use the Lipschitz continuity of the functions $\Phi_\ell$ and 
$\Phi^{-1}_\ell$ with Lipschitz constants $L_\Phi$ and $L_{\Phi^{-1}}$ respectively. Then
\begin{equation}
\begin{split}
|\Psi^\delta(x)-x|&=|\Phi^{-1}_\ell(\gamma^\delta(\Phi_\ell(x)))-\Phi_\ell^{-1}(\Phi_\ell(x))|\\
&\le L_{\Phi^{-1}}|\gamma^\delta(\Phi_\ell(x)))-\Phi_\ell(x)|\le  L_{\Phi^{-1}}\delta
\end{split}
\end{equation}
and 
\begin{equation}
\begin{split}
|\Psi^\delta(x)-y|&=|\Phi^{-1}_\ell(\gamma^\delta(\Phi_\ell(x)))-y|\\
&\ge \frac{1}{L_{\Phi}}|\gamma^\delta(\Phi_\ell(x))-\Phi_\ell(y)|\ge  \frac{\delta}{L_{\Phi}}.
\end{split}
\end{equation}

\bigskip

\section{Selections and convergence in multi-valued terms}\label{selections}
Let $\mA$ be a maximal monotone graph satisfying {(A1)}--{ (A5)}. We call 
a mapping $\tilde A:Q\times \R^{d}\to\R^{d}$ a  selection of $\mA$
if it 
   assigns for a.a. $(t,x)\in Q$ to each  $\bB\in\R^{d}$
exactly one value $\tilde A(t,x,\bB)\in \R^{d}$ such that $(\bB,\tilde A(t,x,\bB))\in\mA(t,x)$.
 One immediately observes that  each such a selection $\tilde A$  is monotone and conditin  {(A4)} implies that  for all $\xi\in
 \R^{d}$ and a.a. $(t,x)\in Q$
\begin{enumerate}
\item[{ (A4$^*$)}]  $\quad\tilde A(t,x,\xi)\cdot\xi\ge -k(x,t)+c_*(M(t,x,|\xi|)+M^*(t,x,|\tilde A(t,x,\xi)|).$ 
\end{enumerate}
Also condition { (A3)} implies for a selection the following property  (see also~\cite{AlAm}) 
\begin{enumerate}
\item[{ (A3$^*$)}] For $(\xi,\bS)\in\R^{d}\times\R^d$:
$${\rm if}\  (\bS-\tilde A(t,x,\bB),\xi-\bB)\ge0 \textrm{ for all } \bB\in\R^{d},
\textrm{ then } (\xi,\bS)\in\mA(t,x).$$
\end{enumerate}
In general, a selection of the graph $\mA$  does not have to be a Borel function, however there is a selection
$\tilde A$ that is a Borel function, see e.g. \cite{AubinFrankowska}, and only such  a selection is here considered.
\begin{lemma}[Properties of $\tilde A$] \label{LS*}
Let $\mathcal{A}(t,x)$ be maximal monotone $M$-graph satisfying { (A1)}--{ (A5)} with measurable selection
$\tilde A:Q\times \mathbb{R}^{d} \to \mathbb{R}^{d}$. Then $\tilde A$
satisfies the following conditions:
\begin{enumerate}
\item [{\it (a1)}]$\mathrm{Dom}\, \tilde A(t,x,\cdot) = \Rd$ a.e. in $Q$;
\item [{\it (a2)}]$\tilde A$ is  monotone, i.e. for every $\xi_1$,
$\xi_2 \in \Rd$ and a.a. $(t,x)\in Q$
\begin{equation} \label{monot}
(\tilde A(t,x,\xi_1) - \tilde A(t,x,\xi_2))\cdot( \xi_1 - \xi_2) \ge 0;
\end{equation}
\item [{\it (a3)}] There are non-negative $k\in
L^1(Q)$,   $c_*>0$ and $N$-function  $M$ such that for all $\xi\in\Rd$ the function
 $\tilde A$ satisfies
\begin{equation} \label{growthS*}
\tilde A \cdot \xi \geq -k(t,x) +c_*(M(t,x,|\xi|) + M^*(t,x,|\tilde A|))
\end{equation}
\end{enumerate}
Moreover, let
$U$ be a dense set in $\Rd$ and
$(\bB,\tilde A(t,x,\bB)) \in \mA(t,x)$ for a.a. $(t,x)\in Q$ and for all $\bB\in U$. Let also
  $(\xi, \bS)\in \Rd \times \Rd$. Then the following conditions are equivalent:
\begin{equation*}\begin{split}
\textrm{(i)} \quad &({\bS} - \tilde A(t,x,\bB))\cdot( {\xi} - \bB) \geq 0 \quad \textrm{ for all } \quad (\bB,\tilde A(t,x,\bB))\in\mA(t,x)\,,\\
\textrm{(ii)} \quad & (\xi, \bS) \in \mA(t,x).
\end{split} 
\end{equation*}
\end{lemma}
For the proof in $L^p$ setting see~\cite{FrMuTa2004}.
Next we shall recall  the proof of Lemma~\ref{Minty2}, cf.~\cite{BuGwMaRaSw2012}. The essence of the presented framework is a generalization of  the Minty method in two directions: to nonreflexive spaces and to maximal monotone graphs. 
\smallskip

\begin{proof} Let $\tilde A$ be a selection of the graph $\mA$ and $(\nabla u^n,A^n)\in\mA.$
The monotonicity  provides that
\begin{equation}\label{43}
(\tilde A(t,x,\bB)-\bS^n))\cdot(\bB-\nabla u^n) \geq 0 \quad \textrm{for all }\, \bB
\in L^{\infty}(Q).
\end{equation}
The limit passage in the term  $\tilde A(t,x,\bB)\nabla u^n$ shall be provided by the uniform integrability of the sequence
$\{\nabla u^n\}$, which is the consequence of the fact that $\nabla u^n\in L_M(Q)$, cf.~Lemma~\ref{uni-int}. To conclude the boundedness of  $\tilde A(t,x,\bB)$ define the set
\begin{equation}
Q_{(K)}:=\{(t,x)\in Q: |k(t,x)|\le K\},
\end{equation}
where  $k\in L^1(Q)$ is the function appearing in the assumption (A4). 
%
Let $\bB\in L^\infty(Q)$ and assume that
 $\tilde A(t,x,\bB)$ is unbounded in $Q_{(K)}$. It follows from {(A4)} and nonnegativity of  $M$  that
$$
|\bB|\ge\frac{ c_*M^\ast(t,x,|\tilde A(t,x,\bB)|)-k(x,t)}{|\tilde A(t,x,\bB)|}.
$$
Since $M^*$ is an $N-$function, then the right-hand side tends to infinity, which contradicts that $\bB$ is bounded. 
Thus,  after integrating \eqref{43} over $Q_{(K)}$ we obtain
\begin{equation}\label{44}
\int_{Q_{(K)}} \bS^n\cdot \nabla u^n \,d x\,d t  \geq \int_{Q_{(K)}}
\bS^n\cdot\bB \,d x\,d t + \int_{Q_{(K)}} \tilde A(t,x,\bB)\cdot(\nabla u^n-\bB)\,d x\,d t.
\end{equation}
Letting  $n\to\infty$ in \eqref{44}, we conclude from \eqref{1.26}--\eqref{Ass} that
\begin{equation}\label{45}
\int_{Q_{(K)}}\bS\cdot \nabla u \,d x\,d t \geq \int_{Q_{(K)}} \bS\cdot\bB \,d x\,d t +
\int_{Q_{(K)}} \tilde A(t,x,\bB)\cdot(\nabla u-\bB)\,d x\,d t
\end{equation}
which we rearrange as follows
\begin{equation}\label{46}
\int_{Q_{(K)}} (\tilde A(t,x,\bB)-\bS)\cdot(\bB-\nabla u)\,d x\,d t \geq 0 \quad \textrm{for all }\, \bB
\in L^{\infty}(Q).
\end{equation}
 For  any $j>0$ we define the set
$Q_{j} := \{z\in Q_{(K)};  |\nabla u(z)|\leq j\}
$
and by $ \bbbone_{Q_{j}}$ we mean  the characteristic function of $Q_j$. Since $|\nabla u| \in L^1(Q)$ due to \eqref{1.26} we observe that
\begin{equation}
|Q\setminus Q_j|\le \frac{C}{j}. \label{smallset}
\end{equation}
For arbitrary  $i,j\in\N, \ 0<j<i$ we choose $B$ in \eqref{46} in the following  form
$$
\bB:=\nabla u \bbbone_{Q_{i}} + h\,W\bbbone_{Q_{j}},\ h>0, \ W\in L^{\infty}(Q).
$$
Thus
\begin{equation}\label{13} \int_{Q_{j}}(\tilde A(t,x,\nabla u+h W)-\bS)\cdot W \,d x\,d t \geq \frac{1}{h}\int_{Q_{(K)} \setminus Q_{i}} \left(\tilde A(t,x,0)\cdot \nabla u - \bS\cdot \nabla u\right) \,d x\,d t.
\end{equation}
For passing to the limit with $i\to\infty$ in \eqref{13} we use Lebesgue dominated convergence theorem concluding from  
\eqref{smallset} and from  
%
\begin{equation}\label{war}
\int_{Q} \left|\tilde A(t,x,0)\cdot \nabla u - \bS\cdot \nabla u\right|\,d x\,d t<\infty
\end{equation}
that
%
\begin{equation}\label{int-on-set}
\lim\limits_{i\to\infty}\int_{Q_{(K)} \setminus Q_{i}}
\left(\tilde A(t,x,0)\cdot \nabla u+\bS\cdot\nabla u\right) \,d x\,d t=0.
\end{equation}
Note that \eqref{war}  easily follows from H\"{o}lder inequality and boundedness  of the terms in appropriate 
Musielak-Orlicz  spaces.
A  direct consequence of \eqref{int-on-set} is  that
\begin{equation}\label{MB1-1}
\int_{Q_{j}}(\tilde A(t,x,\nabla u+h W)-\bS)\cdot W \,d x\,d t\geq 0\quad \textrm{for all }j\in\N\,.
\end{equation}
Let now $h\to 0_+$. Using the definition of $Q_j$ it is easy to see that for a subsequence
\begin{align*}
\tilde A(\cdot,\cdot,\nabla u+h W)&\rightharpoonup \bar{\bS} &&\textrm{weakly in }L^2(Q_j),\\
\nabla u+h W &\to \nabla u &&\textrm{strongly  in }L^2(Q_j),\\
(\nabla u+h W,\tilde A(\cdot,\cdot,\nabla u+h W))&\in {\mathcal A}(t,x) &&\textrm{a.e. in }Q_j.
\end{align*}
We observe that for an arbitrary fixed matrix $\bB\in \mathbb{R}^{d}$ by the monotonicity of the graph
\begin{equation}\label{zB}
\int_{Q_j}(\tilde A(t,x,\nabla u + hW)-\tilde A(t,x,B))\cdot(\nabla u+hW -B)\,dx\,dt\ge 0
\end{equation}
Hence passing to the limit with $h\to0_+$ in \eqref{zB}
 we conclude that 
\begin{equation}\label{zB-lim}
\int_{Q_j}(\bar A-\tilde A(t,x,B))\cdot(\nabla u -B)\ge 0
\end{equation}
which yields by (A3$^*$) that 
%
\begin{equation}
(\nabla u(t,x),\bar\bS(t,x))\in\mA(t,x) \qquad \textrm{ a.e. in }Q_j\,. \label{begr}
\end{equation}
Moreover, since $\nabla u$ is bounded in $Q_j$ then also $\bar\bS$ is bounded in $Q_j$ due to \eqref{begr} and again the properties of an $N-$function. Finally, letting  $h\to 0_+$ in \eqref{MB1-1}, we have
$$
\int_{Q_j}(\bar \bS-\bS)\cdot W \,d x\,d t \ge 0 \qquad \textrm{for all } W\in L^{\infty}(Q_j).
$$
Setting  $ W:=-\frac{(\bar\bS-\bS)}{|\bar\bS-\bS|}\bbbone _{\{\bar\bS\neq\bS\}}$   yields
$$
\int_{Q_j}|\bar\bS-\bS|\,d x\,d t\leq0
$$
and therefore \eqref{begr} implies that  $(\nabla u,\bS)\in\mA(t,x)$ a.e. in $Q_j$. But since $j$ was arbitrary, we use \eqref{smallset} and conclude that
 $(\nabla u,\bS)\in\mA(t,x)$ a.e. in $Q_{(K)}$, and then by the arbitrariness of $K$ we finally conclude that 
  $(\nabla u,\bS)\in\mA(t,x)$ a.e. in $Q$.
 \qed
\end{proof}

We complete this part with a short comment on relations to different approaches 
 to multi-valued problems. Here we want to recall the relation between 
$(t,x)-$dependent maximal monotone graphs and 1-Lipschitz functions. 
Following the argumentation in \cite{FrMuTa2004} and \cite{GwZa2007} one concludes that for each graph satisfying (A1)-(A5)
there exists a function $\Phi: Q\times\R^d\to\R^d$ such that
\begin{equation}
{\cal A}(t,x)=\{(e,d)\in\R^d\times\R^d\,|\,d-e=\Phi(t,x,d+e)\}
\end{equation}
and $\Phi$ satisfies the following conditions:
\begin{enumerate}
\item $\Phi$ is a Carath\'eodory function,
\item $\Phi(t,x,\cdot)$ is a contraction for almost all $(t,x)\in Q$,
\item defining the functions $d,e:Q\times\R^d\to\R^d$ as follows 
\begin{equation}\begin{split}
d(t,x,\xi)&=\frac{1}{2}(\xi+\Phi(t,x,\xi))\\
e(t,x,\xi)&=\frac{1}{2}(\xi-\Phi(t,x,\xi))
\end{split}\end{equation}
the following estimate holds
\begin{equation}
d(t,x,\xi)\cdot e(t,x,\xi)\geq -k(t,x) +c_*(M(t,x,|d(t,x,\xi)|) + M^*(t,x,|e(t,x,\xi)|)),
\end{equation}
\item $\Phi(t,x,0)=0$ for almost all $(t,x)\in Q$.
\end{enumerate}
This connection is essentially used for elliptic and parabolic problems including multi-valued terms, see~\cite{GwZa2007}.

\section{Lipschitz boundary}
In the current section we shall comment on the case of less regular boundaries, namely the case of Lipschitz boundary, where the construction presented in Section~\ref{preliminaries} fails.
Lipschitz regularity of the   boundary provides  there exists a finite family of star-shaped Lipschitz domains $\{\Omega_i\}$ such that (cf.~\cite{Novotny})
$$\Omega=\bigcup\limits_{i\in J}\Omega_i.$$
We introduce the partition of unity $\theta_i$ with $0\le\theta_i\le1,\, \theta_i\in {\cal C}^\infty_c(\Omega_i), \,
{\rm supp} \,\theta_i=\Omega_i, \sum_{i\in J}\theta_i(x)=1$ for $x\in\Omega$.  Then we are dealing with the term
$\nabla (\theta_iu)$ and to provide it is in $L_M(Q)$ we need that both $u\nabla \theta_i$ and 
$\theta_i\nabla u$ are in  $L_M(Q)$.

For this reason we  shall need a kind of Poincar\'e inequality in Musielak--Orlicz spaces.
For each multi-index  $\alpha$ denote by $D^\alpha_x$ the distributional derivative of order $\alpha$ with respect 
to the variable $x$. We define the 
Musielak--Orlicz--Sobolev  space as follows
$$W^{1,x}L_M(\Omega)=\{u\in L_M(\Omega): D^\alpha_xu\in L_M(\Omega),\  \forall\  |\alpha|\le 1\},$$
which is a Banach space with a norm
$$\|u\|_{1,M}=\sum\limits_{|\alpha|\le1}\|D^\alpha_x u\|_M$$
if only \eqref{int} holds and $\inf_{(t,x)\in Q} M(t,x,1)>0$.
 The classical results for embeddings of Orlicz--Sobolev spaces are due to Donaldson and Trudinger, cf.~\cite{DoTr71}. Later the optimal embedding theorem was established by Cianchi in 
\cite{Ci96}. An  interesting extension concerns  anisotropic spaces, cf.~\cite{Cianchi}. The issue of the embedding of Musielak--Orlicz--Sobolev spaces into Musielak--Orlicz spaces was considered by Fan \cite{Fa2012} under the following 
assumptions\footnote{In   \cite{Fa2012} the case of only $x$ dependent modulars was considered and since this is not the 
main concern of the current paper we shall
not extend it for the $(t,x)-$dependent case. Nevertheless, the dependence on $t$ is not essential here.}
\begin{itemize}
\item[(M2)] 
$M(x,a)=M(x,1)a$ for $ x\in\bar \Omega, a\in[0,1]$. 
\end{itemize}
Note that condition (M2) is only a technical assumption.  
Indeed, define $M_1:\bar \Omega\times \R_+\to\R_+$ by 
\begin{equation}
M_1(x,a):=\left\{\begin{array}{lcl}
M(x,1)a&{\rm if}&x\in\bar \Omega, \ a\in[0,1],\\[1ex]
M(x,a)&{\rm if}&x\in\bar \Omega, \ a>1.
\end{array}\right.
\end{equation}
Then in the case of bounded domain $\Omega$ it holds $L_M(\Omega)=L_{M_1}(\Omega)$ and 
$W^{1,x}L_M(\Omega)=W^{1,x}L_{M_1}(\Omega)$, cf.~\cite{Musielak}, hence one could consider $M_1$ instead of $M$.
Note that the assumptions formulated up till now on $M$ are  sufficient for the existence of an inverse function $M^{-1}(x,\cdot)$ to $M(x,\cdot).$ We will use it for the definition of $M^{-1}_*$ as follows
\begin{equation}
M^{-1}_*(x,\xi):=\int_0^\xi\frac{M^{-1}(x,\zeta)}{\zeta^\frac{d+1}{d}}d\zeta\quad
\mbox{for }\ x\in\bar \Omega, \ \xi\ge0.
\end{equation}
Condition (M2) provides that the above function is well defined, strictly increasing and concave  for each $x\in \bar \Omega$ and moreover continuously differentiable for $\xi>1$. Define also 
\begin{equation}
\ell(x):=\lim_{a\to\infty}
M^{-1}_*(x,a).
\end{equation} 
Note that $0<\ell(x)\le \infty$ and define the Sobolev conjugate function of $M$, namely $M_*:\bar \Omega\times\R_+\to\R_+$ as follows
\begin{equation}
M_*(x,a):=\left\{\begin{array}{lcl}
s&{\rm if}&x\in\bar \Omega, \ a\in[0,\ell(x)], \ M^{-1}_*(x,s)=a,\\[1ex]
\infty&{\rm if}&x\in\bar \Omega, \ a\ge \ell(x).
\end{array}\right.
\end{equation}
Observe that $M_*$ is also an $N-$function and for all $x\in\bar \Omega$ $M_*(x,\cdot)\in {\cal C}^1((0,\ell(x)))$. 
Having this notation we shall formulate the next assumptions
\begin{itemize}
\item[(M3)] The function
$\ell:\bar \Omega\to(0,\infty]$ is continuous on $\Omega$ and locally Lipschitz  continuous  on 
${\rm Dom}\ (\ell):=\{x: \ell(x)\in\R\}$. 
\item[(M4)]
$M_*$ is locally Lipschitz  continuous on ${\rm Dom}\ (M_*)$ and there exist positive constants
$\delta_0<\frac{1}{d}, c_0$ and $\ell_0<\min_{x\in\bar \Omega} \ell(x)$  such that for all 
$x\in\bar \Omega$ and $\xi\in[\ell_0,\ell(x))$
\begin{equation}
\left|\frac{\partial M_*(x,\xi)}{\partial x_j}\right|\le c_0(M_*(x,\xi))^{1+\delta_0}, \ j=1,\ldots,d,
\end{equation}
provided $\frac{\partial M_*(x,\xi)}{\partial x_j}, \ j=1,\ldots,d$ exists.
\item[(M5)]
Assume that either $\frac{\partial M(x,\xi)}{\partial \xi}$ exists for all $x\in\bar \Omega$ and $\xi\ge 0$ or the following condition is satisfied uniformly for $(t,x)\in Q$
\begin{equation}
\lim\limits_{\xi\to\infty}\frac{\xi\frac{\partial M(x,\xi)}{\partial \xi^+}}{(M(x,\xi))^{1+\frac{1}{d}}}=0
\end{equation}
where by $\frac{\partial M(x,\xi)}{\partial \xi^+}$ we mean the right derivative of $M(x,\cdot)$ at point $\xi$. 
\end{itemize}

Since the Lipschitz continuity of $M_*$ may not always be immediate to verify, note that the Lipschitz continuity of $M$ on 
$\bar \Omega\times \R_+$ implies the Lipschitz continuity 
of $M_*$ on ${\rm Dom}\ (M_*)$.

\begin{lemma} (Poincar\'e inequality)
Let $M$  be an $N-$function satisfying $(M2)-(M5)$. 
Then

\begin{equation}
\|u\|_M\le c\sum\limits_{j=1}^d\|D_j u\|_M
\end{equation}
for all $u\in W^{1,x}_0L_{M}(\Omega)$.
\end{lemma}

\bigskip
\noindent
{\bf Acknowledgements }\\
The author was supported by the grant IdP2011/000661.


\noindent

\bibliographystyle{abbrv}
\bibliography{genorlicz}

\end{document}